\documentclass[10pt]{amsart}
\usepackage{amsmath,amscd}
\usepackage{amssymb}

\usepackage[all]{xy}

\usepackage[T1]{fontenc}

\usepackage{times}

\author{Liran Shaul}
\address{Department of  Mathematics 
The Weizmann Institute of Science, 
Rehovot 76100, 
Israel}
\email{liran.shaul@weizmann.ac.il}
\newtheorem{thm}[equation]{Theorem}
\newtheorem{cor}[equation]{Corollary}
\newtheorem{prop}[equation]{Proposition}
\newtheorem{lem}[equation]{Lemma}
\theoremstyle{definition}
\newtheorem{dfn}[equation]{Definition}
\newtheorem{rem}[equation]{Remark}

\newtheorem{que}[equation]{Question}

\newcommand{\opn}{\operatorname}

\newcommand{\mrm}[1]{\mathrm{#1}}

\renewcommand{\k}{\Bbbk}

\numberwithin{equation}{section} 

\begin{document}

\title{Reduction of Hochschild cohomology over algebras finite over their center}

\begin{abstract}
We borrow ideas from Grothendieck duality theory to noncommutative algebra, and use them to prove a reduction result for Hochschild cohomology for noncommutative algebras which are finite over their center. This generalizes a result over commutative algebras by Avramov, Iyengar, Lipman and Nayak.
\end{abstract}

\maketitle

\setcounter{section}{-1}
\section{Introduction}

All rings in this note are unital, but not necessarily commutative.
By an $A$-module we shall mean a left $A$-module. If $\k$ is a field, and $A$ is a $\k$-algebra, we will view $A$ as a left $A\otimes_{\k} A^{\opn{op}}$-module, so that $A^{\opn{op}}$ is a right $A^{\opn{op}}\otimes_{\k} A$-module. We will denote by $Z(A)$ be the center of $A$.

In this paper, we study the Hochschild cohomology of a $\k$-algebra $A$, where $\k$ is a field. Given an $A$-bimodule $M$, the $n$-th Hochschild cohmology of $A$ with coefficients in $M$, denoted by $\mrm{HH}^n(A|\k;M)$ is given by $\opn{Ext}^n_{A\otimes_{\k} A^{\opn{op}}}(A,M)$. This has a structure of a left $Z(A)$-module via the left $Z(A)$-action of $M$. Similarly, the $n$-th Hochschild homology $Z(A)$-module of $A$ with coefficients in $M$ is given by $\mrm{HH}_n(A|\k;M):= \opn{Tor}_n^{A\otimes_{\k} A^{\opn{op}}}(A,M)$. The reader is referred to \cite{CE,Lo} for more background on Hochschild homology and cohomology.

In their seminal book \cite{CE}, Cartan and Elinberg showed that when the coefficient bimodule $M$ has a special form, its Hochschild homology and cohomology satisfies a reduction formula which replaces the homological construction over the enveloping algebra by a one over the algebra $A$. More precisely, they have shown (\cite[Proposition IX.4.4]{CE}) that if $A$ is a $\k$-algebra, and $M,N$ are left $A$-modules, then there is an isomorphism
\begin{equation}\label{eqn:classic-c}
\mrm{HH}^n(A|\k;\opn{Hom}_{\k}(M,N)) \cong \opn{Ext}^n_A(M,N),
\end{equation}
and that if $M$ is a left $A$-module and $N$ is a right $A$-module, then there is an isomorphism
\begin{equation}\label{eqn:classic-h}
\mrm{HH}_n(A|\k; M\otimes_{\k} N) \cong \opn{Tor}^n_A(N,M).
\end{equation}
With these formulas in hand, it is natural to ask if similar formulas hold for $\mrm{HH}_n(A|\k;\opn{Hom}_{\k}(M,N))$ and for $\mrm{HH}^n(A|\k ;M\otimes_{\k} N)$. 
A few years ago, this was answered positively in the commutative case. Thus, assume that $A$ is a commutative essentially finite type $\k$-algebra (that is, a localization of a finite type $\k$-algebra). If $f:\k \to A$ is the structure map, then it was shown in \cite[Theorem 4.1]{AILN}, that under some finiteness conditions on $M$, there is a functorial isomorphism
\begin{equation}\label{eqn:com-red-c}
\mrm{HH}^n(A|\k ;M\otimes_{\k} N) \cong \opn{Ext}^n_A(\mrm{R}\opn{Hom}_A(M,f^{!}(\k)),N).
\end{equation}
In addition, it was stated in \cite[Theorem 4.6]{AILN}, and proved in \cite[Theorem 4.1.8]{ILN}, that 
\begin{equation}\label{eqn:com-red-h}
\mrm{HH}_n(A|\k;\opn{Hom}_{\k}(M,N)) \cong \opn{Tor}_n^A(\mrm{R}\opn{Hom}_A(M,f^{!}(\k)),N).
\end{equation}
(actually, the result of \cite{AILN} was more general, allowing $\k$ to be an arbitrary noetherian ring, and then computing derived Hochschild cohomology rather than Hochschild cohomology, but in this paper we will always assume $\k$ to be a field, so we will ignore this issue).

Here, $f^{!}$ is the twisted inverse image pseudofunctor from Grothendieck's duality theory (see the book \cite{RD}). For any essentially finite type $\k$-algebra $A$ with structure map $f:\k \to A$, set $R_A := f^{!}(\k)$. Then the complex $R_A$ is a dualizing complex over $A$ (a notion recalled in Section \ref{sec:background} below).

The main result of this paper (Theorem \ref{thm:main}) is a generalization of (\ref{eqn:com-red-c}) for noncommutative algebras which are finite over their center. 
In our recent paper \cite{Sh}, we gave a new interpretation of the above formulas, which allowed us to deduce many interesting relations between the Hochschild homology and Hochschild cohomology of commutative algebras. This interpretation also led us to Theorem \ref{thm:main} below. Let us explain: first, we showed using adjunction that the above results may be rewritten as (we have switched to derived categorical notation)
\begin{equation}\label{eqn:com-twist-c}
\mrm{R}\opn{Hom}_{A\otimes_{\k} A}(A,M\otimes_{\k} N) \cong \mrm{R}\opn{Hom}_A(\mrm{R}\opn{Hom}_A(M,R_A)\otimes^{\mrm{L}}_A \mrm{R}\opn{Hom}_A(N,R_A),R_A),
\end{equation}
and
\begin{equation}\label{eqn:com-twist-h}
A\otimes^{\mrm{L}}_{A\otimes_{\k} A} (\opn{Hom}_{\k}(M,N)) \cong \mrm{R}\opn{Hom}_A(\mrm{R}\opn{Hom}_A(\mrm{R}\opn{Hom}_A(M,R_A),\mrm{R}\opn{Hom}_A(N,R_A)),R_A).
\end{equation}
Now, given two such $\k$-algebras $A,B$, and a functor 
\[
F:\underbrace{\mrm{D}^{\mrm{b}}_{\mrm{f}}(\opn{Mod} A)\times \mrm{D}^{\mrm{b}}_{\mrm{f}}(\opn{Mod} A) \times \dots \times \mrm{D}^{\mrm{b}}_{\mrm{f}}(\opn{Mod} A)}_n \to \mrm{D}(\opn{Mod} B),
\] 
we defined its twist to be the functor 
$F^{!}:\mrm{D}^{\mrm{b}}_{\mrm{f}}(\opn{Mod} A)\times \mrm{D}^{\mrm{b}}_{\mrm{f}}(\opn{Mod} A) \times \dots \times \mrm{D}^{\mrm{b}}_{\mrm{f}}(\opn{Mod} A) \to \mrm{D}(\opn{Mod} B)$, given by 
\[
F^!(M_1,\dots,M_n) := \mrm{R}\opn{Hom}_B( F( \mrm{R}\opn{Hom}_A(M_1,R_A),\dots, \mrm{R}\opn{Hom}_A(M_n,R_A) ), R_B).
\]
This definition was inspired by Grothendieck's construction of the twisted inverse image functor as a twist of the inverse image functor. Using this definition, and letting $-\otimes^{!}_A-$ and $\opn{Hom}^{!}_A(-,-)$ to be the twists of $-\otimes^{\mrm{L}}_A-$ and $\mrm{R}\opn{Hom}_A(-,-)$ respectively, we see that formulas (\ref{eqn:com-twist-c}) and (\ref{eqn:com-twist-h}) become
\begin{equation}\label{com-twist-formula}
\mrm{HH}^n(A|\k ;M\otimes_{\k} N) \cong H^n(M\otimes^{!}_A N),
\end{equation}
and
\begin{equation}
\mrm{HH}_n(A|\k;\opn{Hom}_{\k}(M,N)) \cong H^n(\opn{Hom}^{!}_A(M,N))
\end{equation}
which bear close resemblance to the classical formulas (\ref{eqn:classic-c}) and (\ref{eqn:classic-h}).

With this idea in hand, we turned to the noncommutative case. In absence of Grothendieck duality theory, a natural replacement of $f^{!}(\k)$ was Van den Bergh's rigid dualizing complex (see Definition \ref{def:rigid} below). This is a complex $R_A$ of bimodules over $A$, determined uniquely by $A$ (if it exists), which when $A$ is commutative and essentially of finite type over $\k$, coincide with $f^{!}(\k)$.
Unlike the commutative case where the functor $-\otimes^{\mrm{L}}_A -$ takes $\mrm{D}^{\mrm{b}}_{\mrm{f}}(\opn{Mod} A)\times \mrm{D}^{\mrm{b}}_{\mrm{f}}(\opn{Mod} A)$ to $\mrm{D}_{\mrm{f}}(\opn{Mod} A)$, in the noncommutative case, if $M$ is a finitely generated left $A$-module and $N$ is a finitely generated right $A$-module, then $N\otimes^{\mrm{L}}_A M$ is just a complex of abelian groups and not of $A$-modules. However, we may naturally endow $M$ with a right $Z(A)$-module structure, and then $-\otimes^{\mrm{L}}_A-$ becomes a functor
$\mrm{D}^{\mrm{b}}_{\mrm{f}}(\opn{Mod} A^{\opn{op}})\times \mrm{D}^{\mrm{b}}_{\mrm{f}}(\opn{Mod} A) \to \mrm{D}(\opn{Mod} Z(A))$. This suggests the following generalization of (\ref{com-twist-formula}) to the noncommutative case:

\begin{thm}
Let $\k$ be a field, and let $A$ be a $\k$-algebra. Assume that $Z(A)$ is an essentially finite type $\k$-algebra, and that $A$ is finite over $Z(A)$. Let $R_{Z(A)}$ be the rigid dualizing complex of the center $Z(A)$, and let $R_A$ be the rigid dualizing complex of $A$. Then for any $M\in \mrm{D}_{\mrm{f}}^{-}(\opn{Mod} A)$, and any $N\in \mrm{D}_{\mrm{f}}^{-}(\opn{Mod} A^{\opn{op}})$ there is a bifunctorial isomorphism
\[
\mrm{R}\opn{Hom}_{A\otimes_{\k} A^{\opn{op}}}(A,M\otimes_{\k} N) \cong \mrm{R}\opn{Hom}_{Z(A)}(\mrm{R}\opn{Hom}_{A^{\opn{op}}}(N,R^{\opn{op}}_A)\otimes^{\mrm{L}}_{A^{\opn{op}}}\mrm{R}\opn{Hom}_A(M,R_A),R_{Z(A)})
\]
in $\mrm{D}(\opn{Mod} Z(A))$.
\end{thm}
This is contained in Theorem \ref{thm:main} below. This theorem contains new information even in the case where $A$ is commutative. See Remark \ref{rem:commutative} for details.

\textbf{Acknowledgments.}
The author would like to thank Amnon Yekutieli for some helpful conversations.

\section{Background on dualizing complexes}\label{sec:background}

Let $A$ be a ring which is both left and right noetherian. We denote by $A^{\opn{op}}$ the opposite ring, we let $\opn{Mod} A$ be the category of left $A$-modules, and we denote by $\mrm{D}(\opn{Mod} A)$ its derived category (see \cite{RD} for background on derived categories). Its full subcategory made of bounded complexes of left $A$-modules is denoted by $\mrm{D}^{\mrm{b}}(\opn{Mod} A)$, while its full subcategory made of complexes of left $A$-modules with finitely generated cohomology is denoted by $\mrm{D}_{\mrm{f}}(\opn{Mod} A)$. Over a commutative ring, the notion of a dualizing complex originated in \cite{RD}. In the noncommutative case, dualizing complexes were first defined in \cite{Ye1}. Let us recall the definition.

\begin{dfn}
Let $\k$ be a field, and let $A$ be a $\k$-algebra which is both left and right noetherian. A complex $R \in \mrm{D}^{\mrm{b}}(\opn{Mod} A\otimes_{\k} A^{\opn{op}})$ is called a dualizing complex over $A$ if the following holds:
\begin{enumerate}
\item $R$ has a finite injective dimension over both $A$ and $A^{\opn{op}}$.
\item $R$ has finitely generated cohomology over both $A$ and $A^{\opn{op}}$.
\item The canonical morphisms $A\to \mrm{R}\opn{Hom}_A(R,R)$ and $A\to \mrm{R}\opn{Hom}_{A^{\opn{op}}}(R,R)$ are isomorphisms in $\mrm{D}^{\mrm{b}}(\opn{Mod} A\otimes_{\k} A^{\opn{op}})$.
\end{enumerate}
\end{dfn}

The justification of the name dualizing complex comes from the following fact (\cite[Proposition 4.2]{Ye2}): If $R$ is a dualizing complex over $A$, and if $M\in \mrm{D}_{\mrm{f}}(\opn{Mod} A)$, then the canonical morphism $M \to \mrm{R}\opn{Hom}_{A^{\opn{op}}}(\mrm{R}\opn{Hom}_A(M,R) ,R)$ is an isomorphism in $\mrm{D}(\opn{Mod} A)$.

\begin{lem}
Let $\k$ be a field, let $A$ be a $\k$-algebra which is both left and right noetherian, let $R$ be a dualizing complex over $A$, let $B$ be a commutative $\k$-algebra, and set $D(-):=\mrm{R}\opn{Hom}_A(-,R)$. Let $M\in \mrm{D}^{-}_{(\mrm{f},)}(\opn{Mod} A\otimes_{\k} B)$, and let $N\in \mrm{D}_{\mrm{f}}(\opn{Mod} A)$. Then there is a bifunctorial isomorphism
\[
\mrm{R}\opn{Hom}_A(M,N) \cong \mrm{R}\opn{Hom}_{A^{\opn{op}}}(D(N),D(M))
\]
in $\mrm{D}(\opn{Mod} B)$.
\end{lem}
\begin{proof}
This is just \cite[Proposition 4.2(2)]{Ye2}, using the fact that $B=B^{\opn{op}}$. There, it is also assumed that $N$ has a $B$-structure, but this is actually not needed.
\end{proof}

\begin{prop}\label{prop:hom-dual}
Let $\k$ be a field, let $A$ be a $\k$-algebra which is both left and right noetherian, let $R$ be a dualizing complex over $A$, let $B$ be a commutative $\k$-algebra, and set $D(-):=\mrm{R}\opn{Hom}_A(-,R)$. Let $M\in \mrm{D}^{\mrm{b}}_{\mrm{f}}(\opn{Mod} A)$, and let $N\in \mrm{D}_{(\mrm{f},)}(\opn{Mod} A\otimes_{\k} B)$. Then there is a bifunctorial isomorphism
\[
\mrm{R}\opn{Hom}_A(M,N) \cong \mrm{R}\opn{Hom}_{A^{\opn{op}}}(D(N),D(M))
\]
in $\mrm{D}(\opn{Mod} B)$.
\end{prop}
\begin{proof}
Set $D^{\opn{op}}(-):=\mrm{R}\opn{Hom}_{A^{\opn{op}}}(-,R)$.
We have that $D(M) \in \mrm{D}^{\mrm{b}}_{\mrm{f}}(\opn{Mod} A^{\opn{op}})$. By \cite[Proposition 4.2(1)]{Ye2}, we have that 
$D(N)\in \mrm{D}_{(,\mrm{f})}(\opn{Mod} B\otimes_{\k} A^{\opn{op}})$. Hence, we may apply the previous lemma for $(A^{\opn{op}},D(N),D(M))$, and obtain that there is a $B$-linear isomorphism
\[
\mrm{R}\opn{Hom}_{A^{\opn{op}}}(D(N),D(M)) \cong \mrm{R}\opn{Hom}_A(D^{\opn{op}}(D(M)),D^{\opn{op}}(D(N))).
\]
Again, by \cite[Proposition 4.2(1)]{Ye2}, there is an $A$-linear isomorphism $M\cong D^{\opn{op}}(D(M))$, and an $A\otimes_{\k} B$-linear isomorphism $N\cong D^{\opn{op}}(D(N))$. Hence, there is a $B$-linear isomorphism $\mrm{R}\opn{Hom}_A(M,N) \cong \mrm{R}\opn{Hom}_{A^{\opn{op}}}(D(N),D(M))$.
\end{proof}

\begin{rem}
The condition that $N\in \mrm{D}_{(\mrm{f},)}(\opn{Mod} A\otimes_{\k} B)$ means that $N$ is a complex of left $A$-modules with finitely generated cohomology, such that $N$ also has the structure of a complex of $B$-modules, and moreover, 
$a \cdot (b\cdot n) = b\cdot (a\cdot n)$ for all $a\in A, b\in B$.
\end{rem}

Recall that if $A$ is a $\k$-algebra, and $M$ is a complex of $A$-bimodules, then $M^{\opn{op}}$ is the complex which is equal to $M$ as a complex over $\k$, and whose $A$-bimodule structure is given by
\[
a \cdot r \cdot b:= bra.
\]

The next definition originated in \cite{VdB}.

\begin{dfn}\label{def:rigid}
Let $\k$ be a field, let $A$ be a $\k$-algebra, and let $R_A$ be a dualizing complex over $A$. Then $R_A$ is a rigid dualizing complex over $A$ if there is an isomorphism
\[
\rho : R_A\to \mrm{R}\opn{Hom}_{A\otimes_{\k} A^{\opn{op}}}(A,R_A\otimes_{\k} R^{\opn{op}}_A)
\]
in $\mrm{D}(\opn{Mod} A\otimes_{\k} A^{\opn{op}})$.
\end{dfn}

\begin{rem}
This definition is ambiguous, let us resolve the ambiguity. The complex $R_A\otimes_{\k} R^{\opn{op}}_A$ is a complex of $A\otimes_{\k} A^{\opn{op}}$-bimodules. The left $A\otimes_{\k} A^{\opn{op}}$-structure is given by the outside structure: $(a\otimes b)\cdot (r\otimes s) = ar\otimes sb$. The right $A\otimes_{\k} A^{\opn{op}}$-structure is given by the inside structure: $(r\otimes s)\cdot (a\otimes b) = ra\otimes bs$.
In the right hand side of the isomorphism $\rho$, we treat $A$ as a left $A\otimes_{\k} A^{\opn{op}}$-module, so we compute
\[
\mrm{R}\opn{Hom}_{A\otimes_{\k} A^{\opn{op}}}({ }_{A\otimes_{\k} A^{\opn{op}}}A,{ }_{A\otimes_{\k} A^{\opn{op}}} (R_A)\otimes_{\k} {(R^{\opn{op}}_A)}_{A\otimes_{\k} A^{\opn{op}}})
\]
Now, a priori, this gives a complex of right $A\otimes_{\k} A^{\opn{op}}$-modules. However, because $A\otimes_{\k} A^{\opn{op}}$ has an involution, as explained in the first paragraph of \cite[Section 5]{Ye2}, such a structure give rise to a left $A\otimes_{\k} A^{\opn{op}}$-structure on this complex in a canonical way.
\end{rem}

\begin{prop}
Let $\k$ be a field, let $A$ be a $\k$-algebra, and let $R_A$ be a rigid dualizing complex over $A$. Then $R^{\opn{op}}_A$ is a rigid dualizing complex over $A^{\opn{op}}$.
\end{prop}
\begin{proof}
This is shown in the proof of \cite[Theorem 8.9]{YZ1}.
\end{proof}

\begin{prop}\label{prop:dual-on-ae}
Let $\k$ be a field, let $C$ be an essentially finite type $\k$-algebra, and let $A$ be a finite $C$-algebra. Then the complex $R_A\otimes_{\k} R^{\opn{op}}_A$ is a rigid dualizing complex over the $\k$-algebra $A\otimes_{\k} A^{\opn{op}}$. The left $A\otimes_{\k} A^{\opn{op}}$-structure is given by the outside structure: $(a\otimes b)\cdot (r\otimes s) = ar\otimes sb$. The right $A\otimes_{\k} A^{\opn{op}}$-structure is given by the inside structure: $(r\otimes s)\cdot (a\otimes b) = ra\otimes bs$.
\end{prop}
\begin{proof}
This is a particular case of \cite[Theorem 8.5]{YZ1}.
\end{proof}

\begin{prop}\label{prop:finite-has-rigid}
Let $\k$ be a field, let $C$ be an essentially finite type $\k$-algebra, and let $A$ be a finite $C$-algebra. Then $C$ has a rigid dualizing complex $R_C$, and $A$ has a rigid dualizing complex $R_A$, which is given by $R_A := \mrm{R}\opn{Hom}_C(A,R_C)$.
\end{prop}
\begin{proof}
The fact that $A$ has a rigid dualizing complex is proved in \cite[Theorem 3.6]{YZ2} (see also \cite[Theorem 8.5.6]{AIL}). The second claim is shown in \cite[Proposition 5.9]{Ye2}. There, $C$ is assumed to be of finite type over $\k$, but once we know that $C$ has a rigid dualizing complex, the proof there generalizes to this case.
\end{proof}

\section{Reduction of Hochschild cohomology}

If $\k$ is a field, and $A$ is a $\k$-algebra, given a complex $M$ of $A$-bimodules, its Hochschild homology and cohomology complexes are given by $A\otimes^{\mrm{L}}_{A\otimes_{\k} A^{\opn{op}}} M$ and $\mrm{R}\opn{Hom}_{A\otimes_{\k} A^{\opn{op}}}(A,M)$ respectively. Following the conventions of \cite[Section 1.1.5]{Lo}, we give these complexes a left $Z(A)$-structure via the left $Z(A)$-structure of $M$ (which is induced from the left $A$-structure of $M$).

We now arrive to the first main result of this text.
To explain it, we first recall the commutative situation from \cite{RD}. Let $\k$ be a field, let $A$, $B$ be essentially finite type commutative $\k$-algebras, and let $f:A\to B$ be a $\k$-algebra map. Then the twist of the inverse image functor $B\otimes^{\mrm{L}}_A -:\mrm{D}(\opn{Mod} A) \to \mrm{D}(\opn{Mod} B)$ is given by $f^{!}(-) := \mrm{R}\opn{Hom}_B(B \otimes^{\mrm{L}}_A \mrm{R}\opn{Hom}_A(-,R_A),R_B)$. If $f$ happens to be a finite ring map, then there is an isomorphism of functors $f^{!}(-) \cong \mrm{R}\opn{Hom}_A(B,-)$. Applying this to the finite $\k$-algebra map $A\otimes_{\k} A\to A$, we obtain that 
\[
\mrm{R}\opn{Hom}_{A\otimes_{\k} A}(A,-) \cong \mrm{R}\opn{Hom}_A(A\otimes^{\mrm{L}}_{A\otimes_{\k} A} \mrm{R}\opn{Hom}_{A\otimes_{\k} A}(-,R_{A\otimes_{\k} A}),R_A). 
\]
In other words: for commutative algebras, Hochschild cohomology is the twist of Hochschild homology. For noncommutative algebras, the map $A\otimes_{\k} A^{\opn{op}} \to A$ is not a ring homomorphism. However, it turns out that the above result remains true. Thus, the next result says that for certain noncommutative algebras, Hochschild cohomology over $A$ is isomorphic to the twist of Hochschild homology over $A^{\opn{op}}$.

\begin{thm}\label{thm:hoc-twist}
Let $\k$ be a field, and let $A$ be a $\k$-algebra which is both left and right noetherian. Assume that $A$ satisfies the following:
\begin{enumerate}
\item $A$ has a rigid dualizing complex $R_A$.
\item The $\k$-algebra $A\otimes_{\k} A^{\opn{op}}$ is also left and right noetherian, and the complex $R_A\otimes_{\k} R^{\opn{op}}_A$ is a dualizing complex over $A\otimes_{\k} A^{\opn{op}}$ with a left structure being the outside structure, and the right structure being the inside structure.
\item There is some complex $R_{Z(A)} \in \mrm{D}^{\mrm{b}}(\opn{Mod} Z(A))$ such that there is an isomorphism $R_A \cong \mrm{R}\opn{Hom}_{Z(A)}(A,R_{Z(A)})$ in $\mrm{D}(\opn{Mod} A\otimes_{\k} A^{\opn{op}})$.
\end{enumerate}
Then for any $M\in \mrm{D}_{\mrm{f}}(\opn{Mod} A\otimes_{\k} A^{\opn{op}})$, there is a functorial isomorphism
\[
\mrm{R}\opn{Hom}_{A\otimes_{\k} A^{\opn{op}}}(A,M) \cong 
\mrm{R}\opn{Hom}_{Z(A)}(A^{\opn{op}}\otimes^{\mrm{L}}_{A^{\opn{op}}\otimes_{\k} A} \mrm{R}\opn{Hom}_{A\otimes_{\k} A^{\opn{op}}}(  M,R_A\otimes_{\k} R^{\opn{op}}_A),R_{Z(A)}).
\]
in $\mrm{D}(\opn{Mod} Z(A))$.
\end{thm}
\begin{proof}
The left $A$-structure of $M$ induces a left $Z(A)$-structure on it which commutes with its $A$-structure. Using this structure and the fact that $R_A\otimes_{\k} R^{\opn{op}}_A$ is a dualizing complex over $A\otimes_{\k} A^{\opn{op}}$, by Proposition \ref{prop:hom-dual}, there is functorial isomorphism
\[
\begin{aligned}
\mrm{R}\opn{Hom}_{A\otimes_{\k} A^{\opn{op}}}(A,M) \cong \\ \mrm{R}\opn{Hom}_{A^{\opn{op}}\otimes_{\k} A}(\mrm{R}\opn{Hom}_{A\otimes_{\k} A^{\opn{op}}}(  M,R_A\otimes_{\k} R^{\opn{op}}_A),\mrm{R}\opn{Hom}_{A\otimes_{\k} A^{\opn{op}}}(A,R_A\otimes_{\k} R^{\opn{op}}_A))
\end{aligned}
\]
in $\mrm{D}(\opn{Mod} Z(A))$.
Since $R_A$ is a rigid dualizing complex, by definition, there is an isomorphism $R_A \cong \mrm{R}\opn{Hom}_{A\otimes_{\k} A^{\opn{op}}}(A,R_A\otimes_{\k} R^{\opn{op}}_A)$ in $\mrm{D}(\opn{Mod} A^{\opn{op}}\otimes_{\k} A)$, so the above is naturally isomorphic to 
\begin{equation}\label{eqn:before-center}
\mrm{R}\opn{Hom}_{A^{\opn{op}}\otimes_{\k} A}(\mrm{R}\opn{Hom}_{A\otimes_{\k} A^{\opn{op}}}(  M,R_A\otimes_{\k} R^{\opn{op}}_A),R_A).
\end{equation}
By assumption (3) above, there is an isomorphism
\[
R_A \cong \mrm{R}\opn{Hom}_{Z(A)}({ }_{A\otimes_{\k} A^{\opn{op}}} A_{Z(A)},R_{Z(A)})
\]
Using our convention that $A$ is a left $A\otimes_{\k} A^{\opn{op}}$-module, the right hand side has a structure of a complex of right $A\otimes_{\k} A^{\opn{op}}$-modules. To make this a left module, we use the fact that $A=A^{\opn{op}}$ as $Z(A)$-modules, and that $R_{Z(A)}$ is a symmetric $Z(A)$-module, so we may write
\[
R_A \cong \mrm{R}\opn{Hom}_{Z(A)}( {A^{\opn{op}}}_{A^{\opn{op}}\otimes_{\k} A} ,R_{Z(A)})
\]
where this is an isomorphism of left $A^{\opn{op}}\otimes_{\k} A$-modules.
Plugging this isomorphism into equation (\ref{eqn:before-center}), we obtain an isomorphism between (\ref{eqn:before-center}) and
\[
\mrm{R}\opn{Hom}_{A^{\opn{op}}\otimes_{\k} A}(\mrm{R}\opn{Hom}_{A\otimes_{\k} A^{\opn{op}}}(  M,R_A\otimes_{\k} R^{\opn{op}}_A),\mrm{R}\opn{Hom}_{Z(A)}(A^{\opn{op}},R_{Z(A)})).
\]

Let $P\to \mrm{R}\opn{Hom}_{A\otimes_{\k} A^{\opn{op}}}(  M,R_A\otimes_{\k} R^{\opn{op}}_A)$ be a K-projective resolution over $(A^{\opn{op}}\otimes_{\k} A)\otimes_{\k} (Z(A))^{\opn{op}}$, and let $R_{Z(A)} \to I$ be a K-injective resolution over $Z(A)$. Then the above is naturally isomorphic to
\[
\opn{Hom}_{A^{\opn{op}}\otimes_{\k} A}({ }_{A^{\opn{op}}\otimes_{\k} A}  (P)_{Z(A)},\opn{Hom}_{Z(A)}({ }_{Z(A)}{A^{\opn{op}}}_{A^{\opn{op}}\otimes_{\k} A},{ }_{Z(A)} I)).
\]
(we explicitly specified the various structures on this complex of left $Z(A)$-modules). Then by the hom-tensor adjunction, there is a left $Z(A)$-linear isomorphism
\[
\opn{Hom}_{A^{\opn{op}}\otimes_{\k} A}(P,\opn{Hom}_{Z(A)}(A^{\opn{op}},I)) \cong \opn{Hom}_{Z(A)}(A^{\opn{op}}\otimes_{A^{\opn{op}}\otimes_{\k} A} P,I).
\]
Since the map $(A^{\opn{op}}\otimes_{\k} A) \to (A^{\opn{op}}\otimes_{\k} A)\otimes_{\k} Z(A)$ is flat, it follows that $P$ is K-flat over $A^{\opn{op}}\otimes_{\k} A$. Hence, as $I$ is K-injective over $Z(A)$, there is a left $Z(A)$-linear isomorphism
\[
\opn{Hom}_{Z(A)}(A^{\opn{op}}\otimes_{A^{\opn{op}}\otimes_{\k} A} P,I) \cong 
\mrm{R}\opn{Hom}_{Z(A)} ( A^{\opn{op}}\otimes^{\mrm{L}}_{A^{\opn{op}}\otimes_{\k} A} \mrm{R}\opn{Hom}_{A\otimes_{\k} A^{\opn{op}}}(  M,R_A\otimes_{\k} R^{\opn{op}}_A),R_{Z(A)}).
\]

The composition of all the left $Z(A)$-linear isomorphisms above is the required isomorphism.
\end{proof}
We will later show (Corollary \ref{cor:holds-for-finite}) that the above theorem holds for algebras finite over their center, but first let us deduce some interesting corollaries from it.
Taking $M=A$, in the above theorem, we obtain the following relations between the Hochschild homology complex and Hochschild cohomology complex:
\begin{cor}
Let $\k$ be a field, and let $A$ be a $\k$-algebra which satisfies the assumptions of Theorem \ref{thm:hoc-twist}. Then the Hochschild cohomology complex of $A$ is the $Z(A)$-linear dual of the Hochschild homology complex of $A^{\opn{op}}$ with coefficients in $R_A$. That is:
\[
\mrm{R}\opn{Hom}_{A\otimes_{\k} A^{\opn{op}}}(A,A) \cong \mrm{R}\opn{Hom}_{Z(A)}(A^{\opn{op}}\otimes^{\mrm{L}}_{A^{\opn{op}}\otimes_{\k} A} R_A, R_{Z(A)}).
\]
\end{cor}
\begin{proof}
This is immediate from Theorem \ref{thm:hoc-twist}, using the rigidity isomorphism 
\[
R_A \cong \mrm{R}\opn{Hom}_{A\otimes_{\k} A^{\opn{op}}}(A,R_A\otimes_{\k} R^{\opn{op}}_A).
\]
\end{proof}

\begin{lem}\label{lem:hochsep}
Let $\k$ be a field, and let $A$ be a $\k$-algebra.
For any $M \in \mrm{D}^{-}(\opn{Mod} A)$ and any $N \in \mrm{D}^{-}(\opn{Mod} A^{\opn{op}})$, there is a bifunctorial isomorphism
\[
A \otimes^{\mrm{L}}_{A\otimes_{\k} A^{\opn{op}}} (M\otimes_{\k} N) \cong N\otimes^{\mrm{L}}_A M
\]
in $\mrm{D}(\opn{Mod} Z(A))$. Here, on both sides the $Z(A)$-structure is induced from the left $A$-structure of $M$.
\end{lem}
\begin{proof}
Replacing $M$ and $N$ by projective resolutions $P$ and $Q$ respectively, it is enough to show that
\[
A \otimes_{A\otimes_{\k} A^{\opn{op}}} (P\otimes_{\k} Q) \cong Q \otimes_A P.
\]
If $P$ and $Q$ are modules, this is shown in \cite[Theorem IX.2.8]{CE}, and the same proof generalizes to complexes.
\end{proof}

\begin{lem}\label{lem:sep}
Let $\k$ be a field, and let $A$ be a $\k$-algebra. Assume $A$ is both left and right noetherian.
Let $M\in \mrm{D}_{\mrm{f}}^{-}(\opn{Mod} A)$, let $N\in \mrm{D}_{\mrm{f}}^{-}(\opn{Mod} A^{\opn{op}})$, and let $R,R' \in \mrm{D}^{\mrm{b}}(\opn{Mod} A\otimes_{\k} A^{\opn{op}})$. Then there is a functorial isomorphism
\[
\mrm{R}\opn{Hom}_{A}(M,R) \otimes_{\k} \mrm{R}\opn{Hom}_{A^{\opn{op}}}(N,R') \cong \mrm{R}\opn{Hom}_{A\otimes_{\k} A^{\opn{op}}}(M\otimes_{\k} N,R\otimes_{\k} R').
\]
in $\mrm{D}(\opn{Mod} A^{\opn{op}}\otimes_{\k} A)$.
\end{lem}
\begin{proof}
This is proved in \cite[Lemma 8.4]{YZ1}. The assumption there is that $M$ and $N$ are actually bounded, but the same proof holds in this more general case.
\end{proof}

We may now prove the main theorem of this note, a generalization of \cite[Theorem 4.1]{AILN} to certain noncommutative algebras:

\begin{thm}\label{thm:main}
Let $\k$ be a field, and let $A$ be a $\k$-algebra which satisfies the conditions of Theorem \ref{thm:hoc-twist}. Then for any $M\in \mrm{D}_{\mrm{f}}^{-}(\opn{Mod} A)$, and any $N\in \mrm{D}_{\mrm{f}}^{-}(\opn{Mod} A^{\opn{op}})$ there is a bifunctorial isomorphism
\[
\mrm{R}\opn{Hom}_{A\otimes_{\k} A^{\opn{op}}}(A,M\otimes_{\k} N) \cong \mrm{R}\opn{Hom}_{Z(A)}(\mrm{R}\opn{Hom}_{A^{\opn{op}}}(N,R^{\opn{op}}_A)\otimes^{\mrm{L}}_{A^{\opn{op}}}\mrm{R}\opn{Hom}_A(M,R_A),R_{Z(A)})
\]
in $\mrm{D}(\opn{Mod} Z(A))$.
\end{thm}
\begin{proof}
By Theorem \ref{thm:hoc-twist}, there is a $Z(A)$-linear isomorphism
\begin{equation}\label{eqn:thm-before-sep}
\mrm{R}\opn{Hom}_{A\otimes_{\k} A^{\opn{op}}}(A,M\otimes_{\k} N) \cong
\mrm{R}\opn{Hom}_{Z(A)}(A^{\opn{op}}\otimes^{\mrm{L}}_{A^{\opn{op}}\otimes_{\k} A} \mrm{R}\opn{Hom}_{A\otimes_{\k} A^{\opn{op}}}(  M\otimes_{\k} N,R_A\otimes_{\k} R^{\opn{op}}_A),R_{Z(A)}).
\end{equation}
By Lemma \ref{lem:sep}, there is a functorial $Z(A)$-linear isomorphism
\[
A^{\opn{op}}\otimes^{\mrm{L}}_{A^{\opn{op}}\otimes_{\k} A} \mrm{R}\opn{Hom}_{A\otimes_{\k} A^{\opn{op}}}(  M\otimes_{\k} N,R_A\otimes_{\k} R^{\opn{op}}_A) \cong
A^{\opn{op}}\otimes^{\mrm{L}}_{A^{\opn{op}}\otimes_{\k} A} (\mrm{R}\opn{Hom}_{A}(M,R_A) \otimes_{\k} \mrm{R}\opn{Hom}_{A^{\opn{op}}}(N,R^{\opn{op}}_A)),
\]
and by Lemma \ref{lem:hochsep} (applied to the ring $A^{\opn{op}}$), there is a functorial $Z(A)$-linear isomorphism
\[
A^{\opn{op}}\otimes^{\mrm{L}}_{A^{\opn{op}}\otimes_{\k} A} (\mrm{R}\opn{Hom}_{A}(M,R_A) \otimes_{\k} \mrm{R}\opn{Hom}_{A^{\opn{op}}}(N,R^{\opn{op}}_A)) \cong 
\mrm{R}\opn{Hom}_{A^{\opn{op}}}(N,R^{\opn{op}}_A)\otimes^{\mrm{L}}_{A^{\opn{op}}} \mrm{R}\opn{Hom}_{A}(M,R_A).
\]
Plugging this to equation (\ref{eqn:thm-before-sep}), we obtain the result.
\end{proof}

\begin{cor}\label{cor:holds-for-finite}
Let $\k$ be a field, let $A$ be a $\k$-algebra, and assume that its center $Z(A)$ is essentially finite type over $\k$ and that $A$ is finite over $Z(A)$. Then Theorems \ref{thm:hoc-twist} and \ref{thm:main} holds for $A$, with $R_{Z(A)}$ being the rigid dualizing complex over $Z(A)$.
\end{cor}
\begin{proof}
Conditions (1), (2) and (3) of the theorem follows from Propositions \ref{prop:dual-on-ae} and \ref{prop:finite-has-rigid}.
\end{proof}

\begin{rem}\label{rem:commutative}
If $A$ is assumed to be commutative and essentially of finite type over $\k$, then it satisfies all the conditions of the theorem. In that case, since $A=A^{\opn{op}} = Z(A)$, the theorem says that 
\[
\mrm{R}\opn{Hom}_{A\otimes_{\k} A}(A,M\otimes_{\k} N) \cong \mrm{R}\opn{Hom}_A(\mrm{R}\opn{Hom}_A(N,R_A)\otimes^{\mrm{L}}_A\mrm{R}\opn{Hom}_A(M,R_A),R_A).
\]
Using the derived tensor hom-adjunction and the fact that $R_A$ is a dualizing complex, the right hand side of this is naturally isomorphic to
\[
\mrm{R}\opn{Hom}_A(\mrm{R}\opn{Hom}_A(M,R_A),N)
\]
which is just \cite[Theorem 4.1]{AILN} in the case where $\k$ is a field. There, under the assumption that $\k$ is a field, it is assumed that $M\in \mrm{D}^{\mrm{b}}_{\mrm{f}}(\opn{Mod} A)$, and that $N \in \mrm{D}(\opn{Mod} A)$. Thus, our theorem gives new information even in the commutative case, allowing $M$ to be only bounded above instead of bounded.
\end{rem}

\begin{rem}
Conditions (1) and (2) of Theorem \ref{thm:hoc-twist} holds for many interesting noncommutative algebras. See for example \cite[Theorem 8.1]{YZ2} and
\cite[Corollary 8.7]{VdB} for condition (1), and \cite[Theorem 8.5]{YZ1} for condition (2).
The author does not know of examples when condition (3) is satisfied, except for the case where $A$ is finite over $Z(A)$.
\end{rem}

\begin{rem}\label{rem:lim-of-thm}
If a $\k$-algebra $A$ satisfies conditions (1) and (2) of Theorem \ref{thm:hoc-twist}, and if Theorem \ref{thm:hoc-twist} holds for $A$ (for some complex $R_{Z(A)} \in \mrm{D}(\opn{Mod} Z(A))$), then condition (3) of the theorem also holds for $A$. Indeed, as we have seen, whenever Theorem \ref{thm:hoc-twist} holds for $A$, Theorem \ref{thm:main} also holds for $A$, and then, by rigidity
\[
R_A \cong \mrm{R}\opn{Hom}_{A\otimes_{\k} A^{\opn{op}}}(A,R_A\otimes_{\k} R^{\opn{op}}_A) \cong \mrm{R}\opn{Hom}_{Z(A)}(\mrm{R}\opn{Hom}_{A^{\opn{op}}}(R^{\opn{op}}_A,R^{\opn{op}}_A)\otimes^{\mrm{L}}_{A^{\opn{op}}} \mrm{R}\opn{Hom}_A(R_A,R_A),R_{Z(A)}),
\]
so that $R_A \cong \mrm{R}\opn{Hom}_{Z(A)}(A,R_{Z(A)})$.
\end{rem}

Taking cohomology on both sides of Theorem \ref{thm:main}, we obtain the following reduction formula for Hochschild cohomology with tensor-decomposable coefficients:
\begin{cor}
Let $\k$ be a field, and let $A$ be a $\k$-algebra which satisfy the conditions of Theorem \ref{thm:hoc-twist}. Then for any $M\in \mrm{D}_{\mrm{f}}^{-}(\opn{Mod} A)$, and any $N\in \mrm{D}_{\mrm{f}}^{-}(\opn{Mod} A^{\opn{op}})$ there is a $Z(A)$-linear isomorphism
\[
\mrm{HH}^n(A|\k;M\otimes_{\k} N) \cong \opn{Ext}^n_{Z(A)}(\mrm{R}\opn{Hom}_{A^{\opn{op}}}(N,R^{\opn{op}}_A)\otimes^{\mrm{L}}_{A^{\opn{op}}}\mrm{R}\opn{Hom}_A(M,R_A),R_{Z(A)})
\]
\end{cor}

\begin{cor}
Let $\k$ be a field, let $A$ be a $\k$-algebra, and assume that its center $Z(A)$ is essentially finite type over $\k$ and that $A$ is finite over $Z(A)$.
Then 
\[
\mrm{R}\opn{Hom}_{A\otimes_{\k} A^{\opn{op}}}(A,A\otimes_{\k} A^{\opn{op}}) \cong \mrm{R}\opn{Hom}_{A^{\opn{op}}}(R^{\opn{op}}_A,A)
\]
\end{cor}
\begin{proof}
By Theorem \ref{thm:main}, we have that 
\[
\mrm{R}\opn{Hom}_{A\otimes_{\k} A^{\opn{op}}}(A,A\otimes_{\k} A^{\opn{op}}) \cong \mrm{R}\opn{Hom}_{Z(A)}(R^{\opn{op}}_A\otimes^{\mrm{L}}_{A^{\opn{op}}} R_A,R_{Z(A)}).
\]
By the derived hom-tensor adjunction,
\[
\mrm{R}\opn{Hom}_{Z(A)}(R^{\opn{op}}_A\otimes^{\mrm{L}}_{A^{\opn{op}}} R_A,R_{Z(A)}) \cong \mrm{R}\opn{Hom}_{A^{\opn{op}}}(R^{\opn{op}}_A,\mrm{R}\opn{Hom}_{Z(A)}(R_A,R_{Z(A)})).
\]
Since $R_A \cong \mrm{R}\opn{Hom}_{Z(A)}(A,R_{Z(A)})$, and since $R_{Z(A)}$ is a dualizing complex over $Z(A)$, we have that
\[
\mrm{R}\opn{Hom}_{Z(A)}(R_A,R_{Z(A)}) \cong \mrm{R}\opn{Hom}_{Z(A)}(\mrm{R}\opn{Hom}_{Z(A)}(A,R_{Z(A)}),R_{Z(A)}) \cong A,
\] 
which proves the claim.
\end{proof}

\begin{rem}
This result was previously proved in \cite[Proposition 5.11]{Ye2} for Gorenstein algebras which possess a rigid dualizing complex.
\end{rem}

\begin{que}
Is there a similar reduction formula for the the Hochschild homology 
\[
\mrm{HH}_n(A|\k;\opn{Hom}_{\k}(M,N))
\]
over noncommutative algebras, as in \cite[Theorem 4.1.8]{ILN} (or \cite[Theorem 4.6]{AILN}), where $M,N$ are complexes of left $A$-modules?
Theorem \ref{thm:main} and the twisting formalism of \cite{Sh}  suggests that a correct form for such a formula might be
\begin{equation}\label{eqn:conj}
A \otimes^{\mrm{L}}_{A\otimes_{\k} A^{\opn{op}}} \opn{Hom}_{\k}(M,N)) \cong \mrm{R}\opn{Hom}_{Z(A)}( \mrm{R}\opn{Hom}_{A^{\opn{op}}}( \mrm{R}\opn{Hom}_A(M,R_A),\mrm{R}\opn{Hom}_A(N,R_A)),R_{Z(A)}).
\end{equation}
The author was not able to prove such a formula, mainly because of the lack of an analogue of Lemma \ref{lem:sep} for complexes of the form $\opn{Hom}_{\k}(M,N)$. 
Taking $M=N=A$ in equation (\ref{eqn:conj}), one will get the following remarkable formula $R_A \cong A \otimes^{\mrm{L}}_{A\otimes_{\k} A^{\opn{op}}} \opn{Hom}_{\k}(A,A)$ for the rigid dualizing complex of $A$. This was shown to be true in the commutative case in \cite[Corollary 4.7]{AILN}.
\end{que}

\end{document}